\documentclass[12pt,reqno]{amsart}
\usepackage{amssymb,amsmath,amsthm,enumerate,enumitem,verbatim,bbm,bm,mathrsfs, mathtools, mathrsfs}

\usepackage[a4paper]{geometry}

\DeclareMathOperator{\supp}{supp}

\DeclareMathOperator{\Id}{Id}

\renewcommand\Im{\hbox{{\rm Im}}\,}
\renewcommand\Re{\hbox{{\rm Re}}\,}

\newcommand{\Abs}[1]{\left\lvert#1\right\rvert}
\newcommand{\norm}[1]{\lVert#1\rVert}

\newcommand{\Br}[1]{\left(#1\right)}
\newcommand{\CBr}[1]{\left\{#1\right\}}

\newcommand{\jap}[1]{\langle#1\rangle}

\newcommand{\C}{{\mathbb C}}
\newcommand{\D}{{\mathbb D}}

\newcommand{\N}{\mathbb N}

\newcommand{\R}{{\mathbb R}}
\newcommand{\T}{{\mathbb T}}

\newcommand{\Z}{{\mathbb Z}}

\newcommand{\bT}{{\mathcal{T}}}

\newcommand{\calE}{\mathcal{E}}

\newcommand{\calN}{\mathcal{N}}

\newcommand{\calX}{\mathcal{X}}

\newcommand{\scrS}{{\mathscr S}}

\DeclareFontFamily{U}{mathx}{\hyphenchar\font45}
\DeclareFontShape{U}{mathx}{m}{n}{<5> <6> <7> <8> <9> <10>
<10.95> <12> <14.4> <17.28> <20.74> <24.88> mathx10}{}
\DeclareSymbolFont{mathx}{U}{mathx}{m}{n}
\DeclareFontSubstitution{U}{mathx}{m}{n}
\DeclareMathAccent{\widecheck}{0}{mathx}{"71}

\numberwithin{equation}{section}

\theoremstyle{plain}
\newtheorem{theorem}{\bf Theorem}[section]
\newtheorem{theoremA}{\bf Theorem}

\newtheorem{lemma}[theorem]{\bf Lemma}
\newtheorem{proposition}[theorem]{\bf Proposition}

\newtheorem{corollary}[theorem]{\bf Corollary}

\theoremstyle{definition}

\theoremstyle{remark}
\newtheorem*{remark*}{\bf Remark}
\newtheorem{remark}[theorem]{\bf Remark}
\newtheorem{example}[theorem]{\bf Example}

\newcommand{\wb}{\overline}
\newcommand{\clos}{\overline}
\newcommand{\wh}{\widehat}

\newcommand{\wt}{\widetilde}
\newcommand{\eps}{\varepsilon}

\newcommand{\1}{\mathbbm{1}}
\newcommand{\dm}{\,\,\mathrm{d}}

\renewcommand{\[}{\begin{equation}}
\renewcommand{\]}{\end{equation}}

\begin{document} 
\title{Completeness of systems of inner functions}
\date{\today} 

\author{Nazar Miheisi}
\address{Department of Mathematics,
King's College London,
Strand, London WC2R 2LS,
United Kingdom}
\email{nazar.miheisi@kcl.ac.uk}



\begin{abstract}
For two inner functions $\vartheta,\varphi\in H^\infty$, we give a simple
sufficient condition for the system $\vartheta^m,\; \varphi^n$, $m,n\in\Z$,
to be complete in the weak-$^*$ topology of $L^\infty(\T)$. To be precise,
we show that this system is complete whenever there is an arc $I$ of the unit
circle $\T$ such that $\vartheta$ is univalent on $I$ and $\varphi$ is univalent
on $\T\setminus I$. As an application of this result, we describe a class of
analytic curves $\Gamma$ such that $(\Gamma, \calX)$ is a Heisenberg uniqueness
pair, where $\calX$ is the lattice cross $\{(m,n)\in\Z^2:\, mn=0\}$. Our main
result extends a theorem of Hedenmalm and Montes-Rodr\'iguez for atomic inner
functions with one singularity \cite{hem1}.
\end{abstract}

\maketitle


\section{Introduction}\label{sec:intro}


Let $\D=\{z\in\C:|z|<1\}$ be the unit disk and $\T=\partial\D$ be the
unit circle. For $1\leq p\leq \infty$, let $H^p(\T)$ denote the usual
Hardy space of functions in $L^p(\T)$ whose Poisson extension to $\D$
is analytic. We recall that a function $\vartheta\in H^\infty(\T)$ is
\emph{inner} if $|\vartheta(\zeta)|=1$ for almost all $\zeta\in\T$.

This paper is concerned with the following problem: for which inner functions
$\vartheta$ and $\varphi$ is the system $\vartheta^m,\; \varphi^n$, $m,n\in\Z$, complete
in the weak-$^*$ topology of $L^\infty(\T)$ -- that is, when does the system
have weak-$^*$ dense linear span? This problem was first considered in the
pioneering work of Hedenmalm and Montes-Rodr\'iguez \cite{hem1} in the special
case of the atomic inner functions
$$
\phi_1(z)=\exp\Br{\lambda_1\frac{z+1}{z-1}},
\quad
\phi_2(z)=\exp\Br{\lambda_2\frac{z-1}{z+1}},
\quad
\lambda_1,\lambda_2>0.
$$
The main result of \cite{hem1} is the following:

\begin{theoremA}[Hedenmalm--Montes-Rodr\'iguez]\label{thm:H-M}
The linear span of the system $\{\phi_1^m,\; \phi_2^n:\, m,n\in\Z\}$
is weak-$^*$ dense in $L^\infty(\T)$ if and only if $\lambda_1\lambda_2\leq\pi^2$.
\end{theoremA}

\noindent
Hedenmalm and Montes-Rodr\'iguez later strengthened this by showing that the
condition $\lambda_1\lambda_2\leq\pi^2$ also characterises the (weak-$^*$)
completeness of $\{\phi_1^m,\;\phi_2^n:\, m,n\geq0\}$ in $H^\infty(\T)$
\cite{hem2, hem3}. It is worth noting that the necessity of this condition is
easily seen and was known earlier; the main point of Theorem \ref{thm:H-M},
and the subsequent $H^\infty(\T)$ version, is the sufficiency.

A related question for the algebra $\C[\phi_1,\phi_2]$ was previously considered
by Mathesson and Stessin \cite{mas1}. In particular, they showed that for
$1\leq p<\infty$, $\C[\phi_1,\phi_2]$ is dense in $H^p(\T)$ if
$\lambda_1\lambda_2<\pi^2$ and not dense if $\lambda_1\lambda_2>\pi^2$. They
left the critical case $\lambda_1\lambda_2=\pi^2$ and the question of weak-$^*$
density in $H^\infty(\T)$ as open problems, both of which were settled by the
Hedenmalm--Montes-Rodr\'iguez theorems (although the latter also follows from
an earlier result of Nazaki \cite[Thm. 4]{Naz1}).


\subsection{Main results}\label{subsec:result}

The condition $\lambda_1\lambda_2\leq\pi^2$ in Theorem \ref{thm:H-M} is
equivalent to the following geometric property of $\phi_1$ and $\phi_2$:
$\T$ can be partitioned into two complementary arcs such that $\phi_1$
is univalent on one of the arcs and $\phi_2$ is univalent on the other.
The main purpose of this paper is to show that if this property is satisfied
by arbitrary inner functions, the conclusion still holds.
Our main theorem is the following:

\begin{theorem}\label{thm:main}
Let $\vartheta$ and $\varphi$ be inner functions of degree at least $3$.
Suppose there is an open arc $I\subseteq\T$ such that
\begin{enumerate}[label=(\arabic*),itemsep=1pt,topsep=0pt]
\item\label{thm:main:cond1}
$\vartheta$ is univalent on $I$ and $\varphi$ is univalent on
$\T\setminus \clos{I}$;
\item\label{thm:main:cond2}
$\vartheta$ and $\varphi$ are continuous at the endpoints of $I$.
\end{enumerate}
Then the linear span of the system $\{\vartheta^m,\,\varphi^n: m,n\in\Z\}$ is
weak-$^*$ dense in $L^\infty(\T)$.
\end{theorem}

Let $BMOA(\T)$ denote the subspace of $H^1(\T)$ consisting of functions
of bounded mean oscillation. It is a classical fact that the
Riesz projection is bounded from $L^p(\T)$ to $H^p(\T)$ for $1<p<\infty$, and
weak-$^*$ continuous from $L^\infty(\T)$ to $BMOA(\T)$. As a consequence,
we get the following corollary:

\begin{corollary}\label{cor:Hp}
Let $\vartheta$ and $\varphi$ satisfy the conditions of Theorem \ref{thm:main}.
Then the linear span of the system $\{\vartheta^m,\,\varphi^n: m,n\geq0\}$ is
dense in $H^p(\T)$ for each $1\leq p<\infty$, and weak-$^*$ dense in $BMOA(\T)$.
\end{corollary}

\noindent
Theorem \ref{thm:main} and Corollary \ref{cor:Hp} naturally prompt the
following question: if $\vartheta$ and $\varphi$ satisfy the conditions of
Theorem \ref{thm:main}, is the linear span of $\{\vartheta^m,\,\varphi^n:
m,n\geq0\}$ weak-$^*$ dense in $H^\infty(\T)$? This is an intriguing problem
which remains open.

For the algebra $\C[\vartheta, \varphi]$ we can say more. Indeed, a result
of Nazaki \cite[Thm. 4]{Naz1} states that any subalgebra of $H^\infty(\T)$
which contains the constants and is dense in $H^2(\T)$ must be weak-$^*$ dense
in $H^\infty(\T)$. Combining this with Corollary \ref{cor:Hp} immediately
yields the following:

\begin{corollary}\label{cor:algebra}
Let $\vartheta$ and $\varphi$ satisfy the conditions of Theorem \ref{thm:main}.
Then the algebra $\C[\vartheta,\varphi]$ is weak-$^*$ dense in $H^\infty(\T)$.
\end{corollary}


\subsection{Heisenberg uniqueness pairs}\label{subsec:HUP}

Let $M(\R^2)$ denote the set of finite Borel measures on $\R^2$. For $\mu\in
M(\R^2)$, the Fourier transform of $\mu$ is
$$
\wh{\mu}(\xi_1,\xi_2)
=
\int_{\R^2} e^{-i \pi(x_1\xi_1 + x_2\xi_2)} \dm\mu(x_1,x_2).
$$
Let $\Gamma$ be a piecewise-smooth curve in $\R^2$ (which need not be connected).
Let $AC(\Gamma)$ be the set of measures in $M(\R^2)$ which are supported on
$\Gamma$ and are absolutely continuous with respect to arc length.
Given a subset $\Sigma$ of $\R^2$, we say that $(\Gamma, \Sigma)$ is a
\emph{Heisenberg uniqueness pair} if $\mu\in AC(\Gamma)$ and
$\wh\mu|_\Sigma=0$ implies $\mu=0$.

The concept of Heisenberg uniqueness pairs was introduced by Hedenmalm and
Montes-Rodr\'iguez \cite{hem1}. Their main result, which is an alternative formulation of Theorem \ref{thm:H-M}, is that if $\Gamma$ is
the graph of $t\mapsto \beta/t$, where $\beta>0$, and $\calX$ is the
\emph{lattice cross}
\[
\calX=\{(m,n)\in\Z^2:\, mn=0\},
\label{eq:cross}
\]
then $(\Gamma, \calX)$ is a Heisenberg uniqueness pair if and only if
$\beta\leq1$. Since then, the problem of exhibiting Heisenberg uniqueness
pairs has attracted considerable attention
(see e.g. \cite{bab1,bag1, gir1,gis1, Goncalves-Ramos1,jak1,lev1,sjo1,sjo2}).

As an application of Theorem \ref{thm:main}, we will show that
$(\Gamma, \calX)$ is a Heisenberg uniqueness pair for a more general class
of curves than the hyperbola. In particular, we will consider the case
where $\Gamma$ is the graph of the function
$$
F_\nu(t) = \int_{-1}^1 \frac{\dm\nu(s)}{t-s},
$$
where $\nu$ is a positive singular measure on $(-1,1)$. Observe that by
taking $\nu=\beta\delta_0$ in the following theorem, we recover the
Hedenmalm--Montes-Rodr\'iguez result for the graph of $t\mapsto\beta/t$.

\begin{theorem}\label{thm:HUP}
Let $\nu\in M(\R)$ be a positive singular measure whose closed support is
contained in $(-1,1)$. Let $\Gamma=\{(t,F_\nu(t)): t\in\R\setminus \supp\nu\}$
and let $\calX$ be the lattice cross \eqref{eq:cross}. If
\[
\int_{-1}^1 \frac{\dm\nu(t)}{1-t^2} \leq 1,
\label{eq:HUP}
\]
then $(\Gamma, \calX)$ is a Heisenberg uniqueness pair.
Moreover, if $\nu$ is even, i.e. $\nu(E)=\nu(-E)$ for each Borel set $E$,
the condition \eqref{eq:HUP} is also necessary for $(\Gamma, \calX)$
to be a Heisenberg uniqueness pair.
\end{theorem}

\begin{proof}
Let us observe that $\mu\in AC(\Gamma)$ if and only
if there exists a function $f\in~L^1(\R,\sqrt{1+|F_\nu'(t)|^2}\dm t)$ such that
for each $m,n\in\Z$,
$$
\wh\mu(m,n)
=
\int_\R e^{-im\pi t}e^{-in\pi F_\nu(t)}\,f(t)\sqrt{1+|F_\nu'(t)|^2}\dm t.
$$
Hence $(\Gamma,\calX)$ is a Heisenberg uniqueness pair precisely if the
linear span of $\{e^{im\pi t},\; e^{-in\pi F_\nu(t)}:\;m,n\in\Z\}$ is
weak-$^*$ dense in $L^\infty(\R)$. By composing with a conformal map from $\D$
to the upper half plane $\C_+$ and using Theorem \ref{thm:main}, we see that
this will hold if $e^{-i\pi F_\nu(t)}$ is univalent on $\R\setminus[-1,1]$ and
continuous at $\pm1$.

The condition on the support of $\nu$ ensures that $F_\nu$ is continuous at
$\pm1$. Observe that that $F_\nu$ is continuous and monotone decreasing on
$\R\setminus(-1,1)$ with $\lim_{t\to\pm\infty}F_\nu(t)=0$. Then if
\eqref{eq:HUP} holds, we have
$$
F_\nu(1)-F_\nu(-1)
= \int_{-1}^1 \frac{2\dm\nu(t)}{1-t^2}
\leq 2,
$$
and so $e^{-i\pi F_\nu(t)}$ is univalent on $\R\setminus[-1,1]$. This proves
the sufficiency.

Suppose $\nu$ is even. Then it is straightforward to check that for all
$x,y\in\R$, $F_\nu(x+iy)=-\overline{F_\nu(-x+iy)}$. In particular, we have that
$F(1+iy)-F(-1+iy)$ is a real-valued continuous function of $y$ which tends
to $0$ as $y\to\infty$. Then if \eqref{eq:HUP} does not hold, so that
$F(1)-F(-1)>2$, there must exist some $y>0$ such that $F(1+iy)-F(-1+iy)=2$.
It follows that the Poisson extensions of $e^{i\pi t}$ and $e^{-i\pi F_\nu(t)}$
do not separate the points of $\C_+$ and hence the linear span of
$\{e^{im\pi t},\; e^{-in\pi F_\nu(t)}:\;m,n\in\Z\}$ is not weak-$^*$ dense
in $L^\infty(\R)$.
\end{proof}


\subsection{Structure of the paper}

In Section \ref{sec:prelim} we begin with some preliminary material on inner
functions and Aleksandrov-Clark measures. We will also introduce an important
operator $T_\vartheta$ associated to an inner function $\vartheta$. The proof
of Theorem \ref{thm:main} will rely on properties of a family of
\emph{generalised counting functions}, which are modelled on the classical
Nevanlinna counting function; these are introduced and studied in Section
\ref{sec:nevanlina}. The proof of Theorem \ref{thm:main} is then carried
out in Sections \ref{sec:transfer} and \ref{sec:span-L}. In Section
\ref{sec:transfer}, we introduce and describe a transfer operator $\bT$ whose
fixed points correspond to functions $f\in L^1(\T)$ satisfying
$\jap{f,\vartheta^m} = \jap{f,\varphi^n}=0$ for all $m,n\in\Z$. Finally, in
Section \ref{sec:span-L}, we complete the proof by combining the results of
Sections \ref{sec:nevanlina} and \ref{sec:transfer} to show that $\bT$ does
not possess any non-zero fixed points.

\subsection{Basic notation and conventions}

Here we fix some more notation and conventions that are used throughout the paper.

\begin{enumerate}[label=\arabic*., leftmargin=*, itemsep=5pt]
\item
The symbol ``$C$'' will denote a positive constant whose precise value may
change from one occurrence to another. Moreover, we write $f(z)\approx g(z)$
if $f(z)\leq Cg(z)$ and $g(z)\leq Cf(z)$ for all (appropriate) $z$, and write
$f(z)\sim g(z)$ as $z\to z_0$ if $\lim_{z\to z_0} f(z)/g(z) =1$. 
\item
The characteristic function of a set $E$ is denoted by $\1_E$, i.e. $\1_E(x)=1$
if $x\in E$ and $\1_E(x)=0$ if $x\notin E$. We simply write $\1$ for a function
which is identically equal to $1$ on its domain.
\item
We use $\C_+$ to denote the upper half plane $\{z\in\C:\Im z>0\}$. Moreover,
we set $\D_+=\D\cap\C_+$, $\;\D_-=\D\setminus\clos{\D_+}$, $\;\T_+=\T\cap\C_+$
and $\T_-=\T\setminus\clos{\T_+}$.
\item
If $E$ is a Lebesgue measurable subset of $\T$, we identify $L^1(E)$ with the
subspace of $L^1(\T)$ consisting of functions that vanish on $\T\setminus E$.
\end{enumerate}


\subsection*{Acknowledgements	}

We are grateful to Oleg Ivrii for suggesting the argument in the proof of
Lemma \ref{lem:counting}\ref{lem:counting2} which allowed us to remove a
restrictive technical assumption in the main theorem.


\section{Preliminaries}\label{sec:prelim}

\subsection{Inner functions}\label{subsec:inner}

Each inner function $\vartheta$ admits a canonical factorisation
$\vartheta= BS$, where
$$
B(z)= \alpha\prod_{j} \frac{\overline{z_j}}{|z_j|}
\frac{z-z_j}{1-\overline{z_j}z}
$$
is a Blaschke product determined by the zeros $\{z_j\}$ of $\vartheta$
and $\alpha\in\T$ (with the convention that $\overline{z_j}/|z_j|=1$
when $z_j=0$), and
$$
S(z)=\exp\Br{-\int_\T \frac{\zeta+z}{\zeta-z}\dm\mu(\zeta)}
$$
is a singular inner function determined by a positive, finite, singular
Borel measure $\mu$ on $\T$. 
In this case, the set
$$
\scrS(\vartheta)=\T\cap\clos{\{z_j\}\cup\supp\mu}
$$
is the intersection of all closed sets $E\subseteq\T$ such that $\vartheta$
extends analytically across $\T\setminus E$. We recall that the degree of
$\vartheta$, denoted $\deg\vartheta$, is $d$ if $\scrS(\vartheta)$ is empty
and $\vartheta$ has $d$ zeros, and equal to $\infty$ otherwise.

For $a\in\D$, let $\omega_a$ denote the automorphism of $\D$ given by
\[
\omega_a(z) = \frac{z-a}{1 - \wb{a}z}.
\label{eq:mobius}
\]
Then if $\vartheta$ is inner, $\omega_a\circ\vartheta$ is also inner.
The exceptional set of $\vartheta$ is the set
$$
\calE_\vartheta
=
\{a\in\D:\; \omega_a\circ\vartheta \;\text{is not a Blashke product}\}.
$$
A celebrated theorem of Frostman asserts that $\calE_\vartheta$ has logarithmic
capacity zero.

At each point $\zeta\in\T\setminus\scrS(\vartheta)$, the angular derivative
of $\vartheta$ -- that is, the derivative of $\arg\vartheta(\zeta)$ with respect
to $\arg\zeta$ -- is given by
\[
|\vartheta'(\zeta)|
=
\sum_{j}\frac{1-|z_j|^2}{|\zeta-z_j|^2} +
\int_\T \frac{2}{|\zeta-\eta|^2} \dm\mu(\eta).
\label{eq:derivative}
\]
We refer to \cite[\S 4.1]{Mashregi1} for the details. It follows that the
continuous branch of $\arg\vartheta(\zeta)$ is strictly increasing with
$\arg\zeta$ in any component of $\T\setminus\scrS(\vartheta)$, and in
particular, that for each $\zeta\in\T\setminus\scrS(\vartheta)$, $\vartheta$
is invertible on a neighbourhood of $\zeta$.

For $\zeta\in\T\setminus\scrS(\vartheta)$, a \emph{local invariant} for
$\vartheta$ (at $\zeta$) is a conformal map $\tau$ defined on some
neighbourhood $U$ of $\zeta$ such that $\vartheta(\tau(z))=\vartheta(z)$
for all $z\in U$
-- that is, $\tau$ is a branch of $\vartheta^{-1}\circ\vartheta$ defined on
a neighbourhood of $\zeta$ on which $\vartheta$ is univalent.
Observe that since $\vartheta$ is locally invertible at each point of
$\T\setminus\scrS(\vartheta)$, we have that for any $\zeta,\eta\in\T\setminus
\scrS(\vartheta)$ with $\vartheta(\zeta)=\vartheta(\eta)$, there is a
local invariant $\tau$ such that $\tau(\zeta)=\eta$.

\subsection{Aleksandrov-Clark measures}\label{subsec:AC-measures}

Let $\vartheta$ be an inner function. For $\alpha\in\T$, the function
$z\mapsto P(\vartheta(z),\alpha)$, where
$$
P(z, \zeta) = \frac{1-|z|^2}{|\zeta-z|^2}
$$
is the Poisson kernel, is a positive harmonic function on $\D$. Thus there
exists a (unique) positive Borel measure $\mu_\alpha$ such that
$$
P(\vartheta(z),\alpha)
=
\int_\T P(z, \zeta) \dm\mu_\alpha(\zeta).
$$
The measures $\{\mu_\alpha:\alpha\in\T\}$ are called the \emph{Aleksandrov-Clark
measures} (AC measures) for $\vartheta$. We mention that AC measures can be
defined in the same way for any function in the unit ball of $H^\infty(\T)$,
but we will only need to consider the case of inner functions in this paper.
We will make use of the following basic properties of AC measures
(see e.g. \cite{MaS2, Saksman1}):
\begin{enumerate}[label=(\arabic*), itemsep=5pt]
\item\label{prop:AC-cts}
The map $\alpha\mapsto \mu_\alpha$ is weak-$^*$ continuous.
\item\label{prop:AC-carrier}
The set $\vartheta^{-1}(\alpha)=\{\zeta\in\T:\lim_{r\to1^-}\vartheta(r\zeta)=\alpha\}$
is a carrier for $\mu_\alpha$, in the sense
that $\mu_\alpha(E)=\mu_\alpha(E\cap\vartheta^{-1}(\alpha))$ for every Borel
set $E\subseteq\T$.
\item\label{prop:AC-atom}
If $\vartheta$
is analytic at $\zeta\in\T$ and $\vartheta(\zeta)=\alpha$, then $\mu_\alpha$
will have an atom at $\zeta$ and $\mu_\alpha(\{\zeta\})=|\vartheta'(\zeta)|^{-1}$.
Consequently, if $\vartheta$ is a Blaschke product of degree $d$, then each
$\mu_\alpha$ will be a linear combination of $d$ distinct point masses.
\item\label{prop:AC-limts}
If $\vartheta$ is not a finite Blaschke product, then $\supp\mu_\alpha$ must
be infinite for each $\alpha\in\T$. It follows from \ref{prop:AC-carrier} that
the limit points of $\supp\mu_\alpha$ must be contained in $\scrS(\vartheta)$.
\end{enumerate}
Many of the deep applications of AC measures are a consequence of the
Aleksandrov disintegration theorem which states that for $f\in L^1(\T)$,
one has that $f\in L^1(\mu_\alpha)$ for almost every $\alpha\in\T$ and
$$
\int_\T \int_\T f(\zeta) \dm\mu_\alpha(\zeta) \dm\alpha
=
\int_\T f(\zeta) \dm\zeta.
$$

\subsection{The Operators $T_\vartheta$}\label{subsec:expectation}

Let $\vartheta$ be an inner function such that $\vartheta(0)=0$ and let
$\{\mu_\alpha:\alpha\in\T\}$ be the AC measures for $\vartheta$. The
orthogonal projection $P_\vartheta$ from $L^2(\T)$ onto the closure of
$\C[\vartheta,\overline{\vartheta}]$ is given by
\[
P_\vartheta f(\zeta)
=
\int_\T f \dm\mu_{\vartheta(\zeta)}.
\label{eq:ac-formula}
\]
The formula \eqref{eq:ac-formula} can be found in, for example, \cite{MaS2}.
Moreover, $P_\vartheta$ extends to a norm $1$ projection from $L^p(\T)$ onto
the closure of $\C[\vartheta,\overline{\vartheta}]$ for every $1\leq p\leq\infty$, 
where the closure is taken with respect to the $L^p$ norm when $p<\infty$
and the weak-$^*$ topology when $p=\infty$. We remark that $P_\vartheta$ is
precisely the conditional expectation corresponding to the $\sigma$-algebra
generated by $\vartheta$.

Suppose $\scrS(\vartheta)\neq\T$. Then for $f\in L^1(\T)$, we define
the function $T_\vartheta f$ by setting
\[
T_{\vartheta} f(\zeta)
=
|\vartheta'(\zeta)|P_\vartheta f(\zeta) - f(\zeta),
\quad
\zeta\in\T\setminus\scrS(\vartheta).
\label{eq:T-op}
\]
It's clear that if $f$ satisfies $\jap{f,\vartheta^n}=0$ for all $n\in\Z$,
then $T_\vartheta f = - f$ almost everywhere on $\T\setminus\scrS(\vartheta)$.
Let $I$ be an arc of $\T$ on which $\vartheta$ is univalent. Then for
$\zeta\in I$, we have that
\[
T_\vartheta f(\zeta)
= |\vartheta'(\zeta)|\int_\T f \dm\mu_{\vartheta(\zeta)} - f(\zeta)
= |\vartheta'(\zeta)|\int_{\T\setminus I} f \dm\mu_{\vartheta(\zeta)}.
\label{eq:T-op-AC}
\]
In particular, this shows that $T_\vartheta f = T_\vartheta(f\1_{\T\setminus I})$
on $I$.

The following Lemma summarizes the basic properties of $T_\vartheta$ that
we need.
\begin{lemma}\label{lem:T-prop}
Let $\vartheta$ be an inner function such that $\vartheta(0)=0$ and
$\scrS(\vartheta)\neq\T$. Let $I$ be an arc of $\T\setminus\scrS(\vartheta)$
on which $\vartheta$ is univalent. Then for each $f\in L^1(\T\setminus I)$,
the following hold:
\begin{enumerate}[label=(\alph*), itemsep=5pt, topsep=3pt]
\item\label{lem:T-prop:triangle}
$|T_\vartheta f|\leq T_\vartheta(|f|)$;
\item\label{lem:T-prop:positive}
if $f\geq0$, then $T_\vartheta f\geq0$;
\item\label{lem:T-prop:cts}
if $f\in C(\clos{\T\setminus I})$, then $T_\vartheta f\in C(\clos{I})$;
\item\label{lem:T-prop:contraction}
$\norm{T_\vartheta f}_{L^1(I)}\leq\norm{f}_{L^1(\T\setminus I)}$;
\item\label{lem:T-prop:adjoint}
if $E\subseteq I$ is Lebesgue measurable, then
$\jap{T_\vartheta f,\1_E}=\jap{f,\1_{\vartheta^{-1}(\vartheta(E))\setminus E}}$.
\end{enumerate}
\end{lemma}

\begin{proof}
\ref{lem:T-prop:triangle} and \ref{lem:T-prop:positive} are obvious, and
\ref{lem:T-prop:cts} follows from the weak$^*$ continuity of the function
$\zeta\mapsto |\vartheta'(\zeta)|(\mu_{\vartheta(\zeta)}-\delta_\zeta)$ on
$\clos{I}$.

To see \ref{lem:T-prop:contraction}, we compute that
$$
\int_I |T_\vartheta f(\zeta)| \dm\zeta
\leq
\int_I|\vartheta'(\zeta)|
\int_{\T\setminus I} |f| \dm\mu_{\vartheta(\zeta)}\dm\zeta
\leq
\int_{\T}
\int_{\T\setminus I} |f| \dm\mu_{\zeta}\dm\zeta.
$$
The conclusion now follows from the Aleksandrov disintegration theorem.

To prove \ref{lem:T-prop:adjoint}, we begin by observing that since
$E\subseteq I$, $\jap{f,\1_E}=0$, and so
$$
\jap{T_\vartheta f,\1_E}
=
\jap{|\vartheta'|P_\vartheta f,\1_E}
=
\jap{f,P_\vartheta (|\vartheta'|\1_E)}.
$$
Then from \eqref{eq:ac-formula} we have
$$
P_\vartheta (|\vartheta'|\1_E)(\zeta)
=
\int_E |\vartheta'| \dm\mu_{\vartheta(\zeta)}.
$$
Since $\vartheta$ is univalent on $E$, we see that
$\mu_{\vartheta(\zeta)}|_E \neq 0$ if and only if there is some $\eta\in E$
such that $\vartheta(\eta)=\vartheta(\zeta)$; in this case
$$
\mu_{\vartheta(\zeta)}|_E = \frac{1}{|\vartheta'(\eta)|}\delta_\eta,
$$
and so $P_\vartheta (|\vartheta'|\1_E)(\zeta)=1$. We conclude that
$P_\vartheta (|\vartheta'|\1_E)= \1_{\vartheta^{-1}(\vartheta(E))}$, from
which the result follows since $\jap{f,\1_E}=0$.
\end{proof}

\begin{example}\label{ex:blaschke}
Let $B$ be a finite Blaschke product of degree $d$. The work of Cassier and
Chalendar \cite{CaC1} shows that there exists a collection $G_B$ of homeomorphisms
of $\T$ with the following properties:
\begin{enumerate}[label=(\arabic*), itemsep=5pt]
\item\label{ex:Blaschke:group}
$G_B$ is a cyclic group (under composition) of order $d$;
\item\label{ex:Blaschke:analytic}
each $g\in G_B$ extends analytically to a neighbourhood of $\T$;
\item\label{ex:Blaschke:invariant}
each $g\in G_B$ satisfies $B(g(\zeta)) = B(\zeta)$ for all $\zeta\in\T$;
\item\label{ex:Blaschke:orbit}
for each $\zeta\in\T$, $B^{-1}(B(\zeta))=\{g(\zeta):g\in G_B\}$.
\end{enumerate}
We call $G_B$ the \emph{group of invariants} of $B$. For $\zeta\in\T$ and
$g\in G_B$ we have $B'(g(\zeta))=g'(\zeta)/B'(\zeta)$, and so
if $\mu_{B(\zeta)}$ is the AC measure for $B$ corresponding to $B(\zeta)$, then
$$
\mu_{B(\zeta)}
=
\frac{1}{|B'(\zeta)|}\sum_{g\in G_B} |g'(\zeta)|\delta_{g(\zeta)}.
$$
Consequently, for $f\in L^1(\T)$ we have
\[
T_B f(\zeta)
=
\sum_{g\in G_B\setminus\{\Id_\T\}} |g'(\zeta)|f(g(\zeta)).
\label{eq:T-blaschke}
\]
\end{example}


\section{Nevanlinna-type counting functions}\label{sec:nevanlina}

\subsection{Nevanlinna counting function}\label{subsec:counting}

Let $\phi:\D\to\D$ be a non-constant analytic function. Its \emph{Nevanlinna
counting function} $N_\phi$ is defined on $\D\setminus\{\phi(0)\}$ by
$$
N_\phi(z) = \sum_{w\in\phi^{-1}(z)}
-\log|w|,
$$
where each $w\in\phi^{-1}(z)$ is repeated according to its multiplicity
and an empty sum is defined to be zero. The Nevanlinna counting function
is a classical tool in value distribution theory and the study of composition
operators. An important property of $N_\phi$ is Littlewood's estimate
(see e.g. \cite[\S 4.2]{sha2}):
\[
N_\phi(z) \leq \log\Abs{\frac{1-\wb{z}\phi(0)}{z-\phi(0)}}.
\label{eq:littlewood}
\]
In particular, if $\phi(0)=0$ then $N_\phi(z) \leq -\log|z|$.
In \cite[\S 4.2]{sha2} it was shown that $\phi$ is inner if and only if there
is equality in \eqref{eq:littlewood} at some point in $\D$, and in this case
there is necessarily equality at all $z\in\D\setminus\calE_\phi$.

There is a close connection between the Nevanlinna counting function and
AC measures. This is most definitively demonstrated is the work of Nieminen
and Saksman \cite{NiS1} in which they introduced a measure-valued refinement
of $N_\phi$ and showed that the AC measures arise as boundary values of the
refined counting function. Among other things, they proved the following
\cite[Cor. 5.2, Prop. 7.1]{NiS1}:

\begin{theorem}[Nieminen--Saksman]\label{thm:Niem-Saks}
Let $\vartheta$ be an inner function and let $\{\mu_\alpha:\alpha\in\T\}$ be
the AC measures for $\vartheta$. Then there is a set
$E_1\subseteq\D$ of logarithmic capacity zero such that for each $\alpha\in\T$
and each $f\in C(\clos{\D})$, we have
$$
\frac{1}{-\log|z|}\sum_{w\in\vartheta^{-1}(z)}-f(w)\log|w|
\to
\int_\T f \dm\mu_\alpha
\quad
\text{as}\;\; z\to\alpha\;\;\text{in}\;\;\D\setminus E_1.
$$
\end{theorem}

\subsection{Generalized counting functions}

Let $u$ be a non-negative function on $\D$ which is harmonic on the complement
of a (possibly empty) discrete set of poles. For such a $u$  and $\zeta\in\T$,
we set
$$
\partial u(\zeta) = \lim_{z\to\zeta}\frac{u(z)}{-\log|z|},
$$
whenever this limit exists. In particular, if $u$ has no poles in $\D_+$
(or $\D_-$), and $u(z)\to0$ as $|z|\to1$ in $\D_+$ (resp. $\D_-$), then
$\partial u$ is well-defined and continuous on $\T_+$ (resp. $\T_-$).

Let $\vartheta$ be an inner function and let $u$ be as above. We define the
generalized counting function $\calN_{\vartheta, u}$ on $\D$ by
$$
\calN_{\vartheta, u}(z)
=\sum_{\substack{w\in\vartheta^{-1}(\vartheta(z)) \\ w\neq z}} u(w).
$$
Note that here the sum is over the `orbit'
$\vartheta^{-1}(\vartheta(z))\setminus\{z\}$ rather than the preimages
$\vartheta^{-1}(z)$; this will be slightly more convenient for our purposes.

The main result of this section is the following lemma, which will be a key
ingredient in the proof of Theorem \ref{thm:main}.

\begin{lemma}\label{lem:counting}
Let $S$ be a discrete set in $\D\setminus\D_-$ and let $u:\D\setminus S\to[0,\infty)$ be
harmonic. Assume further that for $z\in\D_-$,
$$
u(z)\leq -C\log|z|.
$$
Let $\vartheta$ be an inner function with $\vartheta(0)=0$ and
$\deg\vartheta\geq3$. Suppose that $\vartheta$ is univalent on $\T_+$ and
continuous at $\pm1$.
Then, possibly after modifying $\calN_{\vartheta,u}$ on a set of logarithmic
capacity zero, the following holds:
\begin{enumerate}[label=(\alph*), itemsep=5pt]
\item\label{lem:counting1}
for $z\in\D_+$, we have $\calN_{\vartheta,u}(z)\leq -C\log|z|$;
\item\label{lem:counting2}
$\calN_{\vartheta,u}$ is harmonic on $\D\setminus S'$, where
$S'=\vartheta^{-1}(\vartheta(S\cup\{0\}))\cap\clos{\D_-}$;
\item\label{lem:counting3}
$\partial\calN_{\vartheta, u}(\zeta)=T_\vartheta \partial u(\zeta)$
for all $\zeta\in\T_+$, where $T_\vartheta$ is defined by \eqref{eq:T-op}.
\end{enumerate}
\end{lemma}

\begin{remark}
It's clear that we can interchange $\D_+$ and $\T_+$ with $\D_-$ and $\T_-$
respectively in Lemma \ref{lem:counting}, and indeed in all subsequent results
in this Section, and the result will remain true. 
\end{remark}

In order to prove Lemma \ref{lem:counting} we will need three preliminary results.

\begin{lemma}\label{lem:univalence}
Let $\vartheta$ be an inner function which is univalent on $\T_+$. Then
$\vartheta$ is univalent on $\D_+$. If $\deg\vartheta\geq 3$, it is even
univalent across $(-1,1)$.
\end{lemma}

In the proof of Lemma \ref{lem:univalence} we will use the fact that for
$a\in\D$ and $-\pi\leq t_1<t_2 \leq\pi$,
$$
\frac{1}{2\pi}\int_{t_1}^{t_2}
\frac{1-|a|^2}{|e^{it}-a|^2} \dm t
=
\alpha/\pi - (t_2-t_1)/2\pi,
$$
where $\alpha=\arg((e^{it_2}-a)/(e^{it_1}-a))$ is the angle subtended
by the arc $\{e^{it}:t_1<t<t_2\}$ from $a$ \cite[Ex. 3, p.40]{gar1}.

\begin{proof}
Take $\vartheta$ to be univalent on $\T_+$ and suppose there are distinct
points $z_1, z_2\in\D_+$ such that $\vartheta(z_1)=\vartheta(z_2)=a$. By
replacing $\vartheta$ with $\omega_a\circ\vartheta$ if necessary, where
$\omega_a$ is given by \eqref{eq:mobius}, we can suppose $a=0$. Then using
\eqref{eq:derivative} we see that at each $\zeta\in\T$, we have
$$
|\vartheta'(\zeta)|
\geq
\frac{1-|z_1|^2}{|\zeta-z_1|^2}
+
\frac{1-|z_2|^2}{|\zeta-z_2|^2}.
$$
Let $\alpha_1$ and $\alpha_2$ be the angles subtended by
$\T_+$ from $z_1$ and $z_2$ respectively. Observe that since $z_1,z_2\in
\D_+$, we have $\alpha_1,\alpha_2>\pi$. Then the winding number of
$\vartheta|_{\T_+}$ around $0$ is
\begin{align*}
\frac{1}{2\pi}\int_{0}^{\pi}
\Abs{\vartheta'(e^{it})} \dm t
&\geq
\frac{1}{2\pi}\int_{0}^{\pi}
\frac{1-|z_1|^2}{|e^{it}-z_1|^2} \dm t
+
\frac{1}{2\pi}\int_{0}^{\pi}
\frac{1-|z_2|^2}{|e^{it}-z_2|^2} \dm t \\[5pt]
&=
\frac{(\alpha_1 +\alpha_2)}{\pi} - 1 > 1,
\end{align*}
which is a contradiction. Moreover, if $\deg\vartheta\geq 3$, $\vartheta$
will have additional zeros or a non-trivial singular factor, so the first
inequality above will be strict. Hence if $z_1,z_2\in\clos{\D_+}\cap\D$
(we allow the case $z_1=z_2$ if $\vartheta$ has a multiple zero), so that
$\alpha_1,\alpha_2\geq\pi$, we will also get a contradiction. Hence
$\vartheta$ must be univalent across $(-1,1)$ in this case.
\end{proof}

\begin{lemma}\label{lem:log-equiv}
Let $\vartheta$ be an inner function with $\vartheta(0)=0$ and
$\deg\vartheta\geq3$. Suppose that $\vartheta$ is univalent on $\T_+$ and
continuous at $\pm1$. Then for $z\in\D_+$ we have the estimate
\[
-\log|\vartheta(z)| \approx -\log|z|.
\label{eq:log-equiv}
\]
\end{lemma}

\begin{proof}
By Lemma \ref{lem:univalence}, $\vartheta$ does not vanish on $\clos{\D_+}
\setminus\{0\}$ and $\vartheta'(0)\neq0$. It follows that $|z/\vartheta(z)|$ is
bounded on $\D_+$. Moreover, since $\vartheta'$ is bounded on $\D_+$, for
each $z\in\D_+$ we have
$$
1-|\vartheta(z)|
\leq
|\vartheta(z/|z|)-\vartheta(z)|
\leq
C(1-|z|).
$$
Since $|\vartheta(z)|\leq|z|$  by the Schwarz lemma, we have the estimates
$|\vartheta(z)|\approx|z|$ and $1-|\vartheta(z)|\approx1-|z|$. The first of
these implies that $\log|\vartheta(z)|\sim\log|z|$ as $z\to0$, and so
\eqref{eq:log-equiv} holds for all $z\in\D_+$ such that $|z|$ is sufficiently
small. But when $|z|$ (and hence $|\vartheta(z)|$) is bounded away from $0$,
we have the elementary estimates
$$
-\log|z|\approx 1-|z| \approx 1-|\vartheta(z)|\approx -\log|\vartheta(z)|,
$$
and so we conclude that \eqref{eq:log-equiv} holds for all $z\in\D_+$.
\end{proof}

\begin{lemma}\label{lem:cf-convergence}
Let $\vartheta$ and $u$ be as in Lemma \ref{lem:counting}. Let $\{\vartheta_j\}$
be a sequence of inner functions such that each $\vartheta_j$ is
univalent on $\T_+$ and $\vartheta_j\to\vartheta$ uniformly on compact subsets
of $\D$. Then there is a set $E_2\subseteq\D$ of logarithmic capacity zero such
that $\calN_{\vartheta_j,u}(z)\to\calN_{\vartheta,u}(z)$ for all
$z\in\D\setminus E_2$.
\end{lemma}

\begin{proof}
First note that, as a simple consequence of Rouche's theorem, for each $z\in\D$
and $0<r<1$ we have
\[
\lim_{n\to\infty}
\sum_{\substack{w\in\vartheta_j^{-1}(\vartheta_j(z)) \\ |w|\leq r}} u(w)
\; =
\sum_{\substack{w\in\vartheta^{-1}(\vartheta(z)) \\ |w|\leq r}} u(w).
\label{eq:convergence1}
\]
In addition, by Lemma \ref{lem:univalence}, each $\vartheta_j$ is univalent on
$\D_+$ and so for $r$ sufficiently close to $1$ (depending on $z$ but independent
of $j$) the sets $\vartheta_j^{-1}(\vartheta_j(z))\cap\{|w|>r\}$ and
$\vartheta^{-1}(\vartheta(z))\cap\{|w|>r\}$ are contained in $\D_-$.

Take $z\in\D$ such that $\vartheta_j(z)\notin\calE_{\vartheta_j}\cup\vartheta_j(S)$
for each $j$ and $\vartheta(z)\notin\calE_{\vartheta}\cup\vartheta(S)$. Note that
by Frostman's theorem, the complement of all such $z$ has logarithmic capacity
zero. Since we have equality in Littlewood's estimate \eqref{eq:littlewood} for
$N_{\vartheta_j}(\vartheta_j(z))$ and $N_{\vartheta}(\vartheta(z))$, it is
straightforward to check that $N_{\vartheta_j}(\vartheta_j(z))\to
N_{\vartheta}(\vartheta(z))$. Hence for each $\eps>0$, we can choose $r$
sufficiently close to $1$ so that
$$
\limsup_{n\to\infty}
\sum_{\substack{w\in\vartheta_j^{-1}(\vartheta_j(z)) \\ |w|> r}} u(w)
\;\leq
C\limsup_{n\to\infty}
\sum_{\substack{w\in\vartheta^{-1}(\vartheta(z)) \\ |w|\leq r}} -\log|w|
< \eps.
$$
Combining this with \eqref{eq:convergence1} completes the proof.
\end{proof}

\begin{proof}[Proof of Lemma \ref{lem:counting}]
\ref{lem:counting1}
This is a straightforward consequence of Littlewood's estimate 
\eqref{eq:littlewood} and Lemma \ref{lem:log-equiv}. Indeed, for $z\in\D_+$
we have
$$
\calN_{\vartheta, u}(z)
\leq 
\sum_{\substack{w\in\vartheta^{-1}(\vartheta(z)) \\ w\neq z}} -C\log|w|
\leq C
N_\vartheta(\vartheta(z))
\leq 
-C\log|\vartheta(z)|
\leq 
-C\log|z|.
$$

\vspace{5pt}

\ref{lem:counting2}
Observe that $\calN_{\vartheta,u}=\calN_{\omega_a\circ\vartheta, u}$
for each $a\in\D$, and so by replacing $\vartheta$ by $\omega_a\circ\vartheta$
with $a\in\D\setminus\calE_\vartheta$ we may suppose that $\vartheta$ is
a Blaschke product. 

Suppose first that $\vartheta$ is a finite Blaschke product of degree $N$ and let
$C=\{z\in\D:\vartheta'(z)=0\}$. Let $D$ be a disk such that $\clos{D}\subseteq
\D\setminus S'$. Then $u$ is bounded on $\vartheta^{-1}(\vartheta(D))$, and so
$\calN_{\vartheta, u}$ is well-defined and bounded on $D$. In addition, each
point in $D\setminus \vartheta^{-1}(\vartheta(C))$ has a neighbourhood $V$
such that $\vartheta^{-1}(\vartheta(V))$ consists of $N$ disjoint domains
$V_1,\dots,V_N$ each of which is mapped conformally onto $V$ by $\vartheta$.
Define $g_j$ on $V$ by setting $g_j:=(\vartheta|_{V_j})^{-1}\circ\vartheta$.
So for $z\in V$,
$$
\calN_{\vartheta, u}(z)
=
\sum_{j=1}^N u(g_j(z)),
$$
which is harmonic on $V$. Hence $\calN_{\vartheta, u}$ is harmonic and bounded
on $D\setminus \vartheta^{-1}(\vartheta(C))$. Since $\vartheta^{-1}
(\vartheta(C))$ is finite, it is removable. We conclude that
$\calN_{\vartheta, u}$ extends to a harmonic function on $\D\setminus S'$.

If $\vartheta$ is an infinite Blaschke product, let $\{\vartheta_j\}$ be its
sequence of partial products. Then $\vartheta_j\to\vartheta$ uniformly on
compact subsets of $\D$, and since $|\vartheta_j'(\zeta)|$ increases to
$|\vartheta'(\zeta)|$ at each $\zeta\in\T$, $\vartheta_j$ is univalent on
$\T_+$ (and hence on $\D_+$) for every $j$. Let $K$ be a compact subset of
$\D\setminus S$. Then $u$ and $\log|\vartheta_j|$ are bounded on
$\vartheta_j^{-1}(\vartheta_j(K))$ uniformly for all sufficiently large $j$.
Moreover, for each $z\in K$, at most one element of $\vartheta_j^{-1}
(\vartheta_j(z))$ is contained in $\D_+$, and so
$$
\calN_{\vartheta_j, u}(z)
\leq C
(1+N_{\vartheta_j}(\vartheta_j(z)))
\leq C
(1-\log|\vartheta_j(z)|)
\leq C <\infty.
$$
It follows that some subsequence of $\calN_{\vartheta_j, u}$ converges uniformly
on compact subsets of $\D\setminus S'$ to a harmonic function. But by Lemma
\ref{lem:cf-convergence}, $\calN_{\vartheta_j, u}(z)\to\calN_{\vartheta, u}(z)$
for $z$ in the complement of a set of capacity zero.

\vspace{5pt}
\noindent
\ref{lem:counting3}
Let $\{\mu_\alpha: \alpha\in\T\}$ be the AC measures for $\vartheta$.
Fix $\zeta\in\T_+$. Let $K$ be a compact neighbourhood of $\zeta$ in
$\D_+\cup\T_+$ and set $L=\clos{\vartheta^{-1}(\vartheta(K))\setminus K}$.
By Lemma \ref{lem:univalence}, $\vartheta$ is univalent on $\D_+$ and so $L$
is contained in $\D_-\cup\T_-$. Let $f$ be a continuous function on $\clos{\D}$
such that $f(z)=(-\log|z|)^{-1}u(z)$ for $z\in L\cap\D_-$ and $f=0$ on $\D_+$.
In particular, this implies that $f=\partial u$ on $\supp\mu_{\vartheta(\zeta)}
\cap\T_-$, and for $z\in K$,
$$
\calN_{\vartheta,u}(z)
=
\sum_{w\in\vartheta^{-1}(\vartheta(z))} -f(w)\log|w|.
$$
As a result we have that for all $z\in K$,
\[
\frac{\calN_{\vartheta,u}(z)}{-\log|z|} 
=
\CBr{\frac{-\log|\vartheta(z)|}{-\log|z|}}
\CBr{\frac{1}{-\log|\vartheta(z)|}
\sum_{w\in\vartheta^{-1}(\vartheta(z))} -f(w)\log|w|}.
\label{eq:counting-radial}
\]
Next, let $E_1$ and $E_2$ be the exceptional sets given by Theorem
\ref{thm:Niem-Saks} and Lemma \ref{lem:cf-convergence} respectively, and
set $E = \vartheta^{-1}(E_1)\cup E_2$.
Then let $z$ tend to $\zeta$ in $\D\setminus E$. The left hand side
of \eqref{eq:counting-radial} will tend to
$\partial\calN_{\vartheta, u}(\zeta)$. Since $\vartheta$ is continuous
at $\zeta$, we will have that $\vartheta(z)\to\vartheta(\zeta)$. Then by the
Julia-Carath\'eodory theorem (see e.g. \cite[Thm. 4.8]{Mashregi1}) and the
elementary estimate $1-x\sim-\log x$ as $x\to1^-$, the first factor on the
right hand side of \eqref{eq:counting-radial} will converge to
$|\vartheta'(\zeta)|$, and by Theorem \ref{thm:Niem-Saks}, the second factor
will converge to 
$$
\int_\T f \dm\mu_{\vartheta(\zeta)}.
$$
Finally, using \eqref{eq:T-op-AC}, we conclude that
$$
\partial\calN_{\vartheta, u}(\zeta)
=
|\vartheta'(\zeta)|\int_\T f \dm\mu_{\vartheta(\zeta)}
=
|\vartheta'(\zeta)|\int_{\T_-} \partial u \dm\mu_{\vartheta(\zeta)}
=
T_\vartheta \partial u (\zeta).
$$
\end{proof}


\section{The transfer operator $\bT$}\label{sec:transfer}

\subsection{Reduction to a fixed point problem}

Throughout \S\ref{sec:transfer} and \S\ref{sec:span-L}, we fix two inner
functions $\vartheta$ and $\varphi$ satisfying the conditions of Theorem
\ref{thm:main}. Let $I$ be an arc such that conditions \ref{thm:main:cond1}
and \ref{thm:main:cond2} in Theorem \ref{thm:main} hold, and let $\omega$
be an automorphism of $\D$ mapping $I$ onto $\T_+$.
Then $\{\vartheta^m,\,\varphi^n: m,n\in\Z\}$ is (weak-$^*$) complete in
$L^\infty(\T)$ if and only if $\{(\vartheta\circ\omega)^m,\,
(\varphi\circ\omega)^n: m,n\in\Z\}$ is complete, so we will suppose that
\ref{thm:main:cond1} and \ref{thm:main:cond2} hold with $I=\T_+$. In
addition, replacing $\vartheta$ and $\varphi$ with
$\omega_{\vartheta(0)}\circ\vartheta$ and $\omega_{\varphi(0)}\circ\varphi$
respectively does not change the closed span of $\{\vartheta^m,\,\varphi^n:
m,n\in\Z\}$ and so we will additionally suppose that $\vartheta(0)=\varphi(0)=0$.
Finally, we let $\{\mu_\alpha:\alpha\in\T\}$ and $\{\nu_\alpha:\alpha\in\T\}$
be the AC measures for $\vartheta$ and $\varphi$ respectively.

Let $T_{\vartheta}$ and $T_{\varphi}$ be as in \eqref{eq:T-op}. Then from
the discussion in \S\ref{subsec:expectation} we see that if $f\in L^1(\T)$
satisfies
$$
\jap{f,\vartheta^m} = \jap{f,\varphi^n}=0, \quad m,n\in\Z,
$$
then $f|_{\T_+} = T_{\vartheta} (\1_{\T_-}f)$ and $f|_{\T_-} =
T_{\varphi} (\1_{\T_+}f)$. Hence $f|_{\T_+} = \1_{\T_+}T_{\vartheta}
(\1_{\T_-}T_{\varphi}(f|_{\T_+}))$, and if $f|_{\T_+}=0$, then $f=0$.
In light of this, we introduce the operator
$$
\bT:L^1(\T_+)\to L^1(\T_+), \quad \bT=\1_{\T_+}T_{\vartheta} \1_{\T_-}T_{\varphi}.
$$
We have shown that $\{\vartheta^m,\; \varphi^n:\; m,n\in\Z\}$ is complete in the
weak-$^*$ topology of $L^\infty(\T)$ if and only if the operator $\bT$ has
no non-zero fixed points.

\begin{remark}
When $\clos{\vartheta(\T_+)}=\clos{\varphi(\T_-)}=\T$, $\bT$ is the
\emph{Perron-Frobenius} or \emph{transfer} operator of a dynamical
system on $\T_+$. Indeed, since $\vartheta$ and $\varphi$ are invertible on
$\T_+$ and $\T_-$ respectively, the map $\phi:\T_+\to\T_+$ given by
$$
\phi=(\vartheta|_{\T_+})^{-1}\circ\vartheta
\circ(\varphi|_{\T_-})^{-1}\circ\varphi
$$
is a well-defined (almost everywhere) nonsingular transformation of $\T_+$.
Let $C_\phi$ be the composition operator $f\mapsto f\circ\phi$. Then $\bT
= C_\phi^*$. Hence existence (or uniqueness) of a non-zero fixed point for
$\bT$ is equivalent to existence (resp. ergodicity) of an absolutely continuous
invariant measure for $\phi$. This point of view was used to great effect in
the work of Hedenmalm, Montes-Rodr\'iguez and Canto-Mart\'in
\cite{hem1, hem2, hem3, chm1}, and has inspired the approach we have taken here.
\end{remark}

\subsection{Properties of $\bT$}

Given a Lebesgue measurable set $E\subseteq\T_+$, we set
$$
\Lambda(E)
=
\varphi^{-1}(\varphi(\vartheta^{-1}(\vartheta(E))\cap\T_-))\cap\T_+.
$$
We collect some elementary properties of $\bT$ in the following lemma.

\begin{lemma}\label{lem:P-properties}
For each $f\in L^1(\T_+)$, the following statements hold:
\begin{enumerate}[label=(\alph*), itemsep=5pt]
\item\label{lem:P-properties:triangle}
$|\bT f| \leq \bT(|f|)$;
\item\label{lem:P-properties:positive}
if $f\geq 0$ then $\bT f\geq 0$;
\item\label{lem:P-properties:cts}
if $f\in C(\clos{\T_+})$ then $\bT f\in C(\clos{\T_+})$;
\item\label{lem:P-properties:contraction}
$\norm{\bT f}_{L^1(\T_+)}\leq\norm{f}_{L^1(\T_+)}$;
\item\label{lem:P-properties:adjoint}
if $E\subseteq \T_+$ is Lebesgue measurable, then
$\jap{\bT f,\1_E}=\jap{f,\1_{\Lambda(E)}}$;
\item\label{lem:P-properties:modulus}
if $\bT f=f$ then $\bT(|f|)=|f|$.
\end{enumerate}
\end{lemma}

\begin{proof}
\ref{lem:P-properties:triangle}--\ref{lem:P-properties:adjoint} follow immediately
from the corresponding properties for $T_{\vartheta}$ and $T_{\varphi}$ given in
Lemma \ref{lem:T-prop}. To prove \ref{lem:P-properties:modulus} we suppose
$\bT f=f$ and observe that
$|f|=|\bT f|\leq\bT(|f|)$. Then integrating both sides and using
\ref{lem:P-properties:contraction} gives
$$
\norm{f}_{L^1(\T_+)}\leq\norm{\bT(|f|)}_{L^1(\T_+)}\leq\norm{f}_{L^1(\T_+)}.
$$
Hence $\norm{f}_{L^1(\T_+)}=\norm{\bT(|f|)}_{L^1(\T_+)}$. But since
$|f|\leq\bT(|f|)$, we must actually have $|f|=\bT(|f|)$ almost everywhere.
\end{proof}

We will also need the following:

\begin{lemma}\label{lem:bdd-below}
The following statements hold:
\begin{enumerate}[label=(\alph*), itemsep=5pt]
\item\label{lem:bdd-below:analytic}
Let $f\in L^1(\T_+)$, $f\neq0$, be analytic on $\T_+$. Then there exists an
integer $n\geq1$ such that $\inf_{\zeta\in\T_+} \bT^n(|f|)(\zeta) > 0$.
\item\label{lem:bdd-below:arc}
There exists a closed arc $J_0\subseteq\T_+$ such that if $f\in L^1(\T_+)$ and
$\inf_{\zeta\in J_0} |f(\zeta)| > 0$, then
$
\inf_{\zeta\in\T_+} \bT(|f|)(\zeta) > 0.
$
\end{enumerate}
\end{lemma}

\begin{proof}
\ref{lem:bdd-below:analytic}
Let us first suppose that $\vartheta$ and $\varphi$ are both finite Blaschke products
and let $G_{\vartheta}$ and $G_{\varphi}$ be the groups of invariants of $\vartheta$ and $\varphi$
respectively (see Example \ref{ex:blaschke}). For $n\geq1$, set
\begin{align*}
\Gamma_1 &= \{h\circ g:\; g\in G_{\vartheta},\; h\in G_{\varphi},\; g,h\neq\Id_\T\}, \\
\Gamma_n &= \{\gamma_1\circ \gamma:\;\gamma_1\in\Gamma_1,\; \gamma\in\Gamma_{n-1}\},
\quad n\geq2.
\end{align*}
Then each $\gamma\in\cup_{n\geq1}\Gamma_n$ is analytic on $\T_+$. Moreover, for
each $n$, $\Gamma_n$ contains $(\deg\vartheta-1)^n(\deg\varphi-1)^n$ elements and the
arcs $\{\gamma(\T_+):\gamma\in\Gamma_n\}$ are pairwise disjoint. Then from the
discussion in Example \ref{ex:blaschke} one easily sees that for
$g\in L^1(\T_+)$, we have
$$
\bT^n g(\zeta)
=
\sum_{\gamma\in\Gamma_n} |\gamma'(\zeta)|g(\gamma(\zeta)).
$$
Thus, by taking $g=|f|$, it suffices to find $n\geq1$ and $\gamma\in\Gamma_n$
such that $f$ has no zeros on $\clos{\gamma(\T_+)}$. Since the zeros of $f$ can
only accumulate at $\pm1$, there is some $\gamma_*\in\Gamma_1$ such that
$\clos{\gamma_*(\T_+)}$ contains finitely many, say $N$, zeros of $f$. Then
taking $n$ sufficiently large so that $\Gamma_n$ contains at least $2N+1$
elements, we see that certainly one of the arcs
$\{\clos{\gamma_*(\gamma(\T_+))}:\gamma\in\Gamma_n\}$
will not contain any zeros of $f$. This proves the claim in the case that
that $\vartheta$ and $\varphi$ are both finite Blaschke products.

Next we suppose that $\varphi$ is not a finite Blaschke product so that
$\scrS(\varphi)$
is non-empty. Take $h\in C(\clos{\T_+})$ such that $0\leq h\leq |f|$ and $h=|f|$ on
an open arc containing $\scrS(\varphi)$. Then since $\bT(|f|)\geq \bT h \geq 0$
and $\bT h\in C(\clos{\T_+})$, it suffices to show that $\bT h$ does not vanish
at any point of $\clos{\T_+}$. Let us first observe that $T_{\varphi} h>0$ on
$\clos{\T_-}$. Indeed, if $T_{\varphi} h(\eta)=0$ for some $\eta\in\clos{\T_-}$,
then $h$ must vanish identically on $\supp\nu_{\varphi(\eta)}$. But since
$\supp\nu_{\varphi(\eta)}$ has an accumulation point in $\scrS(\varphi)$, and $h=|f|$
on a neighbourhood of $\scrS(\varphi)$, this cannot happen. Similarly, since
$T_{\varphi} h$ does not vanish on $\supp\mu_{\vartheta(\zeta)}$ for any
$\zeta\in\clos{\T_+}$, we conclude $\bT h = T_{\vartheta}
\1_{\T_-}T_{\varphi} h > 0$ on $\clos{\T_+}$. For the case when $\varphi$ is a
finite Blaschke product but $\vartheta$ is not, we observe that $T_\varphi f$
is analytic on $\T_-$ (this follows, for example, from the representation 
\eqref{eq:T-blaschke}) and then use the argument above to conclude that
$\inf\bT |f|>0$.

\ref{lem:bdd-below:arc}
In the case that $\vartheta$ and $\varphi$ are both finite Blaschke products,
we can take $J_0$ to be $\clos{\gamma(\T_+)}$, where $\gamma\in\Gamma_1$ is
chosen so that $\clos{\gamma(\T_+)}\subseteq\T_+$. Otherwise we can take $J_0$
to be any closed arc of $\T_+$ which contains a neighbourhood of
$$\scrS(\varphi)\cup\varphi^{-1}(\varphi(\scrS(\vartheta)))\cap\T_+.$$
In either case, the arguments above show that $J_0$ has the desired property.
\end{proof}

\subsection{Local behaviour near an endpoint of $\T_+$}\label{subsec:endpoints}

Assume that $\clos{\vartheta(\T_+)}=\clos{\varphi(\T_-)}=\T$ so that
$\vartheta(1)=\vartheta(-1)$ and $\varphi(1)=\varphi(-1)$. Then from the
discussion at the end of \S\ref{subsec:inner} we see that there are local
invariants $\tau$ and $\sigma$ for $\vartheta$ and $\varphi$ respectively,
such that $\tau(1)=-1$ and $\sigma(-1)=1$. Let $U$ be a fixed neighbourhood
of $1$ on which $\lambda=\sigma\circ\tau$ is a well-defined conformal map,
and set $J=U\cap\T_+$. Observe that for each $E\subseteq J$,
$\lambda(E)\subseteq\Lambda(E)$.

\begin{lemma}\label{lem:neutral-fp}
Suppose there exists $f\in L^1(\T_+)$ such that $f>0$ almost everywhere and
$\bT f=f$. Then $\lambda(J)\subseteq J$ and $|\lambda'(1)|=1$.
\end{lemma}

\begin{proof}
Take an arc $E\subseteq J$ with endpoint $1$. Then by Lemma
\ref{lem:P-properties}\ref{lem:P-properties:adjoint}, we see that
\[
\jap{f,\1_E} = \jap{\bT f,\1_E}
= \jap{f,\1_{\Lambda(E)}}
\geq\jap{f,\1_{\lambda(E)}}.
\label{eq:lambda}
\]
Observe that since $\lambda(1)=1,$ we must have either $\lambda(E)\subseteq E$
or $E\subsetneq\lambda(E)$; since $f>0$ almost everywhere, \eqref{eq:lambda}
implies it must be the former. This shows that $\lambda(J)\subseteq J$ and
$|\lambda'(1)|\leq1$.

Let us temporarily introduce $\wt\lambda=\sigma^{-1}\circ\tau^{-1}$. This is
well-defined and conformal on some neighbourhood of $-1$ and satisfies
$\wt\lambda(-1)=-1$. Then the same arguments above also imply that
$|\wt\lambda'(-1)|\leq1$. However a simple computation shows that
$\lambda'(1)\wt\lambda'(-1)=1$, and so we must in fact have
$|\lambda'(1)|=|\wt\lambda'(-1)|=1$.
\end{proof}

\begin{lemma}\label{lem:local}
Let $f\in L^1(\T_+)$ be supported on the arc $\lambda(J)$. Then for almost
every $\zeta\in J$ we have
$$
\bT f(\zeta)= |\lambda'(\zeta)|f(\lambda(\zeta)).
$$
\end{lemma}

\begin{proof}
Fix $\zeta\in J$. Since $\vartheta(\zeta)=\vartheta(\tau(\zeta))$ and $\vartheta$ is
analytic across $\tau(\zeta)$ we have 
$$
\mu_{\vartheta(\zeta)}|_{\tau(J)}
=
\frac{\delta_{\tau(\zeta)}}{|\vartheta'(\tau(\zeta))|}.
$$
Similarly, for $\eta\in\tau(J)$,
$$
\nu_{\varphi(\eta)}|_{\lambda(J)}
=
\frac{\delta_{\sigma(\eta)}}{|\varphi'(\sigma(\eta))|}.
$$
Moreover, since $\varphi$ is univalent on $\T_-$ and $\varphi(\tau(J))=
\varphi(\lambda(J))$, if $\eta\in\T_-\setminus\tau(J)$ then
$\varphi(\eta)\notin\varphi(\lambda(J))$. Hence $\nu_{\varphi(\eta)}|_{\lambda(J)}=0$
in this case. Then for $f\in L^1(\T_+)$ with $\supp f\subseteq\lambda(J)$, we
have $\supp (f\circ\sigma)\subseteq\tau(J)$, and therefore
\begin{align*}
\bT f(\zeta)
&=
|\vartheta'(\zeta)|\int_{\T_-} |\varphi'(\eta)|\int_{\T_+}
f \dm\nu_{\varphi(\eta)} \dm\mu_{\vartheta(\zeta)}(\eta) \\
&=
|\vartheta'(\zeta)|\int_{\T_-} |\sigma'(\eta)|
f(\sigma(\eta)) \dm\mu_{\vartheta(\zeta)}(\eta) \\
&=
|\tau'(\zeta)| |\sigma'(\tau(\zeta))|f(\sigma(\tau(\zeta)) \\
&=
|\lambda'(\zeta)|f(\lambda(\zeta)).
\end{align*}
\end{proof}


\section{Proof of Theorem \ref{thm:main}}\label{sec:span-L}

\subsection{Absence of analytic fixed points}\label{subsec:analytic}

The proof of Theorem \ref{thm:main} will be complete once we show that $\bT$
has no non-zero fixed points in $L^1(\T_+)$. We begin by showing that
$\bT$ has no non-zero \emph{analytic} fixed points in $L^1(\T_+)$.

\begin{proposition}\label{prop:analytic}
Let $f\in L^1(\T_+)$ be analytic on $\T_+$. If $\bT (f)=f$, then $f=0$.
\end{proposition}

To simplify notation, throughout \S\ref{subsec:analytic} we will
adopt the convention that all set equalities are up to null sets -- that is,
for Lebesgue measurable sets $A,B\subseteq\T$ we write $A=B$ if the symmetric
difference of $A$ and $B$ has Lebesgue measure zero.

\begin{lemma}\label{lem:full-measure}
$\Lambda(\T_+) = \T_+$ if and only if $\vartheta(\T_+)=\varphi(\T_-)=\T$.
\end{lemma}

\begin{proof}
Recall that
$
\Lambda(\T_+) = \varphi^{-1}(\varphi(\vartheta^{-1}(\vartheta(\T_+))\cap\T_-))\cap\T_+.
$
Observe that since $\varphi$ is univalent on $\T_-$ and $\deg\varphi\geq2$ we must
have $\varphi(\T_+)=\T$. It follows that for $E\subseteq\T_-$,
$\varphi^{-1}(\varphi(E))\cap\T_+=\T_+$ if and only if $\varphi(E)=\T$, and since
$\varphi$ is univalent on $\T_-$, this holds if and only if $E=\T_-$ and
$\varphi(\T_-)=\T$. Taking $E=\vartheta^{-1}(\vartheta(\T_+))\cap\T_-$ we see
that $\Lambda(\T_+) = \T_+$ if and only if $\varphi(\T_-)=\T$ and
$\vartheta^{-1}(\vartheta(\T_+))\cap\T_-=\T_-$. Then the same argument above
shows that the second equality holds if and only if $\vartheta(\T_+)=\T$.
\end{proof}

Before continuing, we introduce the following notation: for a set $S$ and
a function $f:S\to S$ we denote the $n\textsuperscript{th}$ iterate of $f$
by $f^{\jap{n}}$, i.e
$$
f^{\jap{n}} = \underbrace{f\circ f\circ\cdots\circ f}_{\text{$n$ terms}}.
$$

\begin{lemma}\label{lem:calculus}
Fix $\beta>0$. Let $f:[0,\beta)\to[0,\beta)$ be an increasing analytic
function such that $f(0)=0$ and $f'(0)=1$. Then for each $0<t<\beta$, there
exists $C=C(t)>0$ such that for each $N\geq2$,
$$
\sum_{n=1}^N f^{\jap{n}}(t) \geq C \log N.
$$
\end{lemma}

\begin{proof}
Since $f$ is analytic, there exists $A>0$ such that $|t-f(t)|\leq At^2$ for
all sufficiently small $t>0$. Take $t>0$ such that this holds and suppose that
for some integers $j\geq 4A+1$ and $k\geq1$ we have $1/j\leq t<1/(j-1)$ and
$f(t)\leq 1/(j+k)$. Then it follows that
$$
1/j-1/(j+k) \leq t-f(t) \leq A/(j-1)^2 \leq 4A/j^2,
$$
and so $k\leq 4Aj/(j-4A)$. Since $\sup_{j\geq4A+1} 4Aj/(j-4A)$ is finite we
see that there exists $K\in\N$ such that if $t\geq1/j$ for some $j\geq 4A+1$,
then $f(t)\geq 1/(j+K)$. We conclude that for any $0<t<\beta$ and
$j\geq\max\{1/t,\,4A+1\}$ we have $f^{\jap{n}}(t) \geq 1/(j+nK)$,
from which the claim follows.
\end{proof}

\begin{proof}[Proof of Proposition \ref{prop:analytic}]
Fix $f\in L^1(\T_+)$ such that $f$ is analytic on $\T_+$ and $\bT f=f$.
Then by Lemma \ref{lem:P-properties}\ref{lem:P-properties:modulus} we also
have $\bT(|f|)=|f|$. The remainder of the proof is split into two cases.

\textbf{Case 1:} Either $\vartheta(\T_+)\neq\T$ or $\varphi(\T_-)\neq\T$.

In this case, Lemma \ref{lem:full-measure} implies that $\Lambda(\T_+)\neq\T_+$.
However, an application of Lemma
\ref{lem:P-properties}\ref{lem:P-properties:adjoint} gives
$$
\norm{f}_{L^1(\T_+)}
= \jap{|f|,\1} = \jap{\bT (|f|), \1}
= \jap{|f|,\1_{\Lambda(\T_+)}}.
$$
Therefore $f$ must vanish on $\T_+\setminus\Lambda(\T_+)$ which has positive
measure. We conclude that $f=0$.

\textbf{Case 2:} $\vartheta(\T_+)=\varphi(\T_-)=\T$.

Assume towards a contradiction that $f\neq0$. 
Let $\lambda$ and $J$ be as in \S\ref{subsec:endpoints} and let $J_0$ be an
arc satisfying the conclusion of Lemma \ref{lem:bdd-below}\ref{lem:bdd-below:arc}. 
By taking $J$ sufficiently small, we can suppose that 
$J_0\subseteq\T_+\setminus\lambda(J)$.
Fix $\zeta\in J$ and consider the following expansion for $|f(\zeta)|$, which
is a consequence of Lemma \ref{lem:local}:
$$
|f(\zeta)|
=
|\lambda'(\zeta)||f(\lambda(\zeta))|
+
\bT \wt f(\zeta),
$$
where $\wt f=|f|\1_{\T_+\setminus\lambda(J)}$.
By Lemma \ref{lem:neutral-fp}, $\lambda(\zeta)\in J$, and so we can substitute
the analogous expansion for $|f(\lambda(\zeta))|$ into the first term on the
right hand side to get
$$
|f(\zeta)|
=
|(\lambda^{\jap{2}})'(\zeta)||f(\lambda^{\jap{2}}(\zeta))|
+
|\lambda'(\zeta)|\bT \wt f(\lambda(\zeta))
+
\bT \wt f(\zeta).
$$
Iterating this process we see that for all $N\geq1$ we have
$$
|f(\zeta)|
=
|(\lambda^{\jap{N}})'(\zeta)||f(\lambda^{\jap{N}}(\zeta))|
+
\sum_{j=0}^{N-1}
|(\lambda^{\jap{j}})'(\zeta)|\bT \wt f(\lambda^{\jap{j}}(\zeta)).
$$
By Lemma
\ref{lem:bdd-below}\ref{lem:bdd-below:analytic} we must have
$|f|\geq C\1$. Moreover, since $J_0\subseteq\T_+\setminus\lambda(J)$,
we have $\wt f\geq C\1_{J_0}$. Then Lemma
\ref{lem:bdd-below}\ref{lem:bdd-below:arc} implies that $\bT \wt f\geq C\1$.
We conclude that for each $N\geq1$,
\[
|f(\zeta)|
\geq
\sum_{j=0}^{N-1}
|(\lambda^{\jap{j}})'(\zeta)|\bT \wt f(\lambda^{\jap{j}}(\zeta))
\geq
C\sum_{j=0}^{N-1}
|(\lambda^{\jap{j}})'(\zeta)|.
\label{eq:lower-bound}
\]

Write $J=\{e^{it}:0<t<\beta\}$ and define $h:[0,\beta)\to[0,\beta)$ by
$h(t)=-i\log\lambda(e^{it})$. One easily checks that $h$ satisfies the
conditions of Lemma \ref{lem:calculus} and $(h^{\jap{j}})'(t)=
|(\lambda^{\jap{j}})'(e^{it})|$. Fix $0<\delta<\beta$. We
now conclude from \eqref{eq:lower-bound} and Lemma \ref{lem:calculus} that
for each $N\geq2$,
\begin{multline*}
\norm{f}_{L^1(\T_+)}
\geq
\frac{1}{2\pi}\int_0^{\delta} |f(e^{it})|\dm t
\geq
C\sum_{j=0}^{N-1}\int_0^{\delta}
|(\lambda^{\jap{j}})'(e^{it})| \dm t \\
=
C\sum_{j=0}^{N-1} h^{\jap{j}}(\delta)
\geq
C \log N.
\end{multline*}
This contradiction shows that we must in fact have $f=0$.
\end{proof}

\subsection{Decay of the Ces\`{a}ro averages}

We will complete the proof of Theorem \ref{thm:main} by showing that the
Ces\`{a}ro averages of $\bT$ converge to zero in a suitably weak sense,
and therefore $\bT$ cannot have a non-zero fixed point.

\begin{proposition}\label{prop:ergodic}
For each $f\in L^1(\T_+)$, and each $\phi\in C(\T_+)$ with compact support,
we have
$$
\frac{1}{N}\sum_{j=0}^{N-1} \jap{\bT^j f,\phi}
\to 0
\quad\text{as}\quad N\to\infty.
$$
\end{proposition}

\begin{lemma}\label{lem:normal}
There exists a sequence of analytic functions $\{F_j\}$, defined on $\C_+$, such
that the following hold:
\begin{enumerate}[label=(\alph*), topsep=1pt, itemsep=2pt]
\item\label{lem:normal:1}
for each $j\geq0$, $F_j=\bT^j\1$ almost everywhere on $\T_+$;
\item\label{lem:normal:2}
the convex hull of $\{F_j\}$ is a normal family.
\end{enumerate}
\end{lemma}

\begin{proof}
We start by defining define two sequences of functions on $\D$, $\{u_j\}$
and $\{v_j\}$, as follows:
\begin{align*}
u_0(z) &= -\log|z|,\\
v_j(z) &= \calN_{\varphi, u_j}(z),\\
u_{j+1}(z) &= \calN_{\vartheta, v_j}(z).
\end{align*}
Here, we are taking the convention that each $\calN_{\varphi, u_j}$ and
$\calN_{\vartheta, v_j}$ have been modified if necessary so that they satisfy
the conclusions of Lemma \ref{lem:counting}. In particular, this means that
each $u_j$ is harmonic on $\D_+$ and $u_j(z)\to0$ as $|z|\to1$ in $\D_+$.
We also have from Lemma \ref{lem:counting} that for each $j\geq0$,
$$
\partial u_{j+1}
=\1_{\T_+}T_{\vartheta}(\partial v_j)
=\1_{\T_+}T_{\vartheta} \1_{\T_-}T_{\varphi}(\partial u_j)
=\bT \partial u_j.
$$ 
Since $\partial u_0 = \1$, we conclude that $\partial u_j=\bT^j\1$ for
all $j\geq0$.

Using the Schwarz reflection principle, for each $j\geq0$, we extend $u_j$
to a real-valued harmonic function on $\C_+$. Let $f_j:\C_+\to\C$ be the
unique analytic function satisfying $\Re f_j=u_j$ and $f_j(i)=0$. It is a simple
consequence of the Cauchy-Riemann equations, that for all $\zeta\in\T_+$,
$$
f_j'(\zeta)
=
\wb\zeta \partial u_j(\zeta)
=
\wb{\zeta}\bT^j\1(\zeta).
$$
Then setting $F_j(z)=zf'_j(z)$, we see that \ref{lem:normal:1} holds.

It remains to show that the convex hull of $\{F_j\}$ is normal. Clearly, it
is sufficient to show that the convex hull of $\{f_j\}$ is normal. To do
this, we observe that each $f_j$ maps $\D_+$ to the right half plane and
$\C_+\setminus\clos{\D}$ to the left half plane. It follows that $f'_j$
cannot vanish on $\T_+$ and so $f_j$ maps $\T_+$ monotonically (in the sense
that $t\mapsto f_j(e^{it})$ is monotonic) onto an interval of $i\R$. Moreover,
$$
|f_j(-1)-f_j(1)|
= \int_0^\pi |f_j'(e^{it})|\dm t
= \int_0^\pi |\bT^j\1(e^{it})|\dm t
\leq \pi,
$$
and so $f_j(\T_+)\subseteq\{iy:-\pi<y<\pi\}$. Clearly, this must also
be true of each convex combination of functions in $\{f_j\}$. The claim now
follows from Montel's theorem.
\end{proof}

\begin{proof}[Proof of Proposition \ref{prop:ergodic}.]
Let $\{F_j\}$ be the sequence of functions given by Lemma \ref{lem:normal}.
First, let us show that the Ces\`{a}ro averages
$$
\frac{1}{N}\sum_{j=0}^{N-1} \bT^j\1 = 
\frac{1}{N}\sum_{j=0}^{N-1} F_j
$$
converge uniformly on compact subsets to zero as $N\to\infty$. By Krengel's
stochastic ergodic theorem \cite[Thm. 4.9]{kre1}, these converge in measure
on $\T_+$ to some function $F\in L^1(\T_+)$ with $\bT F=F$. By Lemma
\ref{lem:normal}, there is a subsequence which converges uniformly on compact
subsets, and moreover, that $F$ extends to an analytic function on
$\C_+$. Finally, by Proposition \ref{prop:analytic} we must have $F=0$.

To complete the proof we take $f\in L^\infty(\T_+)$ and a compactly supported
function $\phi\in C(\T_+)$.
Let $K$ be the support of $\phi$. Then using the positivity of $\bT$ we see that
$$
\Abs{\frac{1}{N}\sum_{j=0}^{N-1} \jap{\bT^j f,\phi}}
\leq
\norm{f}_{\infty}\norm{\phi}_{\infty}
\frac{1}{N}\sum_{j=0}^{N-1}\jap{\bT^j\1,\1_K}
\to 0
\quad\text{as}\quad N\to\infty.
$$
This proves the claim in the case that $f\in L^\infty(\T_+)$. Since $\bT$ is a
contraction, the result follows for all $f\in L^1(\T_+)$ by a simple
approximation argument.
\end{proof}


\bibliographystyle{plain}
\bibliography{references}

\end{document}